\documentclass[a4paper,12pt,reqno,oneside]{amsart}
\usepackage{setspace} 
\usepackage{typearea} % European margins

\usepackage{amscd,amsmath,amssymb,verbatim}
\usepackage{xr-hyper}
\usepackage[mathscr]{eucal} % Красивые буквы для чума
\usepackage{graphicx}
\usepackage{ifthen}
\usepackage{paralist,cite}
\usepackage[colorlinks,final,bookmarksnumbered,bookmarks]{hyperref}
\usepackage{amsrefs}
\usepackage{enumitem}
\usepackage{ulem}
\usepackage{tikz-cd}
\usepackage{accents}

\usepackage[utf8]{inputenc}
\usepackage[T1,T2A]{fontenc}
\usepackage[russian]{babel}
\usepackage{pgfplots}
\usepackage{colortbl}
\usepackage{bold-extra}
\pgfplotsset{compat=1.9}
 
\newcommand{\precdot}{\mathrel{\prec\hspace{-5.5pt}\cdot}}
\def\R{\Bbb{R}}
\def\Z{\Bbb{Z}}
\newcommand\K{\mathcal{K}}
\newcommand\A{\mathcal{A}}
\renewcommand\P{\mathscr{P}}
\theoremstyle{plain}
\newtheorem{theorem}{Теорема}[section]
\newtheorem{lemma}[theorem]{Лемма}
\newtheorem{corollary}[theorem]{Следствие}
\newtheorem{proposition}[theorem]{Предложение}

\theoremstyle{definition}
\newtheorem{definition}[theorem]{Определение}
\newtheorem{example}[theorem]{Пример}

\newtheoremstyle{remark}
{}{}{}{}{\itshape}{}{ }{\thmname{#1}\thmnumber{ \itshape #2.}}
\theoremstyle{remark}
\newtheorem{remark}[theorem]{Замечание}

\DeclareMathOperator{\Aff}{Aff}

\makeatletter
\def\@settitle{\begin{center}%
    \baselineskip14\p@\relax
    \bfseries
    \@title
  \end{center}%
}
\makeatother

\begin{document}

\title[]{О ЧИСЛЕ ГРАНЕЙ МЕЧЕНО-ПОРЯДКОВЫХ МНОГОГРАННИКОВ}
\author{Е. В. Мелихова}

\begin{abstract}
В работе предложен новый способ вычисления 
$f$-вектора мечено-порядкового многогранника.
А именно, для произвольного (полиэдрального) подразбиения произвольного выпуклого многогранника 
мы строим коцепной комплекс (над полем вычетов $\Z_2$), такой что размерности его когомологий совпадают с компонентами $f$-вектора исходного многогранника. Для мечено-порядкового многогранника и его известного кубо-сиплициального подразбиения этот коцепной комплекс удаётся описать чисто комбинаторно --- что и даёт вычисление $f$-вектора.  
Независимый интерес может представлять предложенное в работе комбинаторное описание вышеупомянутого кубо-симплициального подразбиения  (которое исходно было построено  геометрически).
\bigskip

\noindent\textsc{Abstract.}
In this paper, we present a new method for computing the $f$-vector of a marked order polytope.
Namely, given an arbitrary (polyhedral) subdivision of an arbitrary convex polytope,
we construct a cochain complex (over the two-element field $\Z_2$) such that the dimensions of its cohomology groups equal the components of the $f$-vector of the original polytope. In the case of a marked order polytope and its well-known cubosimplicial subdivision, this cochain complex can be described purely combinatorially --- which yields the said 
computation of the $f$-vector.
Of independent interest may be our combinatorial description of the said cubosimplicial subdivision (which was originally constructed geometrically).
\end{abstract}

\address{Мелихова Екатерина Владимировна,
Национальный исследовательский университет <<Высшая школа экономики>>,
Факультет математики,
ул. Усачёва, 6, Москва, 119048, Россия}
\address{Ekaterina V. Melikhova,
National Research University ``Higher School of Economics'',
Faculty of Mathematics,
Usacheva St., 6, Moscow, 119048, Russia}
\email{ekmelikhova86@gmail.com}
\keywords{Marked order polytope, $f$-vector, polyhedral complex.\\
{\it Ключевые слова:} мечено-порядковый многогранник, $f$-вектор, полиэдральный комплекс.}
\subjclass{52B05. УДК 514.172.45}
\maketitle

\section*{Введение}
Если $\P=(P,\preceq)$ -- конечное частично-упорядоченное множество с наименьшим и наибольшим элементами $\hat 0$ и $\hat 1$, то его {\it порядковый многогранник} $O(\P)$ лежит в конечномерном векторном пространстве $\R^P$ всех отображений $x$ из $P$ в прямую. В этом пространстве он задаётся  неравенствами $x(p) \leq x(q)$ для всех $p$, $q$ из $P$, таких что $p \preceq q$, а также равенствами  $x(\hat 0)=0$, $x(\hat 1)=1$. В работе \cite{S} Р. Стенли кроме прочего описал решётку граней многогранника $O(\P)$, а также <<каноническую>> триангуляцию этого многогранника. 

Ф.\ Ардила, Т.\ Блим и Д.\ Салазар в работе \cite{ABS} рассмотрели обобщение многогранника $O(\P)$, построенное по конечному частично-упорядоченному множеству $\P$ с отмеченными элементами (включающими все экстремальные элементы $\P$), на которых задана сохраняющая порядок функция $\lambda$. Таким образом появилось понятие {\it мечено-порядкового многогранника} $O(\P,\lambda)$ (см.\ определение \ref{MOP}). 
Одним из важных примеров мечено-порядковых многогранников являются многогранники Гельфанда--Цетлина.

Мы изучаем комбинаторику мечено-порядковых многогранников.
В частности, хотим описать $f$-вектор такого многогранника (т.е.\ конечную последовательность $(f_0,f_1, ...,f_n)$, где $f_i$ --- число $i$-мерных граней, $n$ --- размерность многогранника).
Для решения этой задачи иногда удобнее работать с соответствующей $f$-вектору производящей функцией 
$f(t)=f_0+f_1\cdot t+\dots+ f_n \cdot t^n$, получившей название $f$-\textit{многочлена}.
В случае многогранников Гельфанда--Цетлина корейскими математиками \cite{BYJ}, а также независимо автором настоящей работы \cite{ME} было найдено рекуррентное соотношение на $f$-многочлен.
Случай вершин был рассмотрен уже в работе \cite{GKT} (см.\ также \cite{LMcA}).

 %Автор получила упомянутый результат, развивая идеи и методы работы \cite{GKT}, посвящённой изучению комбинаторики числа вершин многогранников Гельфанда--Цетлина. 

В случае произвольных мечено-порядковых многогранников для вычисления $f$-век\-то\-ра возможен <<лобовой>> подход, основанный на описании решётки граней $O(\P,\lambda)$, которое дал К. Пегель \cite{CP}. 
В настоящей работе предлагается другой, менее очевидный подход.
В \S\ref{complex-section} для произвольного (полиэдрального) подразбиения $\K$ произвольного выпуклого многогранника $M$ 
мы строим коцепной комплекс $C^*_\K$ (над $\Z_2$), такой что размерности его когомологий совпадают с компонентами $f$-вектора многогранника $M$ (см.\ лемму \ref{lem1}).
В случае мечено-порядковых многогранников эту конструкцию удаётся довести до вычисления $f$-вектора.
А именно, у мечено-порядкового многогранника $O(\P,\lambda)$ имеется {\it кубо-симплициальное} подразбиение $\K_{\P,\lambda}$ (каждый элемент этого подразбиения 
является произведением симплексов), которое обобщает <<каноническую>> триангуляцию Стенли порядкового многогранника $O(\P)$.
Для многогранника Гельфанда--Цетлина это кубо-симплициальное подразбиение построил А. Постников \cite{Pos}*{proof of Theorem 15.1}, а в общем случае --- Р. Лиу, К. Месарош и Э. Сен-Дизье \cite{LMS}*{proof of Theorem 3.4 (см.\ также \S5.2)}.
Из этих работ можно извлечь и комбинаторное описание граней старшей размерности подразбиения $\K_{\P,\lambda}$.
В \S\ref{KPlambda-section} мы описываем грани {\sl всех} размерностей подразбиения $\K_{\P,\lambda}$ в чисто комбинаторных терминах: они соответствуют некоторым цепям вложенных идеалов $\P$ (см.\ предложение \ref{KP} и следствие \ref{KP_OPlambda}). 
Пользуясь этим описанием $\K_{\P,\lambda}$, в \S\ref{fvector-section} мы даём комбинаторное описание коцепного комплекса $C^*_{\K_{\P,\lambda}}$. Полученное комбинаторное описание $C^*_{\K_{\P,\lambda}}$ позволяет вычислить когомологии этого коцепного комплекса, размерности которых, согласно лемме \ref{lem1}, совпадают с компонентами $f$-вектора многогранника $O(\P,\lambda)$.   
Таким образом, мы получаем некоторый новый способ вычисления $f$-вектора мечено-порядкового многогранника $O(\P,\lambda)$ (см.\ теорему \ref{main}).

\section{Многогранник и коцепной комплекс, который <<считает>> его $f$-вектор}\label{complex-section}
Относительную границу многогранника $Q$, то есть его границу в собственной аффинной оболочке, будем обозначать через $\partial Q$, а его относительную внутренность --- через  $\mathring Q$.

В основном следуя \cite{G}, дадим несколько определений.
 
\begin{definition} \label{complexdef}
\textit{Полиэдральным комплексом}  $\K$ называется конечное семейство непустых полиэдров, лежащих в некотором евклидовом пространстве и удовлетворяющих следующим условиям:
\begin{enumerate}
\item любая грань\footnote{Пустое множество мы не считаем гранью.} полиэдра из $\K$ снова принадлежит $\K$;
\item пересечение любых двух полиэдров $M_1, M_2 \in \K$ либо пусто, либо является гранью каждого из них.
\end{enumerate}
Множество $|\K|=\bigcup_{M \in \K} M$ называется \textit{телом} полиэдрального комплекса $\K$.
\end{definition}

\begin{definition}
\textit{Комплексом $\K(M)$ многогранника} $M$ называется полиэдральный комплекс его граней.
\end{definition}

\begin{definition}
Полиэдральный комплекс $\K'$ называется \textit{подразбиением} полиэдрального комплекса $\K$, если $|\K'|=|\K|$ и каждый полиэдр комплекса $\K'$ содержится в некотором полиэдре комплекса $\K$.
\end{definition}

\begin{definition}
\textit{Подразбиением} выпуклого многогранника $M$ называется полиэдральный комплекс $\K$, тело которого совпадает с многогранником $M$ (условие, что каждый многогранник комплекса $\K$ содержится в некоторой грани многогранника $M$ в данном случае выполнятся автоматически). 
\end{definition}

Нам также понадобится понятие \textit{подкомплекса}.

\begin{definition}
Подмножество полиэдрального комплекса называется \textit{подкомплексом}, если оно само является полиэдральным комплексом.
\end{definition}

%\begin{example}
%Пусть $M'$ --- клеточное разбиение выпуклого многогранника $M$, $A$ --- грань многогранника $M$. Определим $A'$ как подмножество в $M'$, состоящее из таких многогранников $Q\in M'$, что $Q\cap A \neq \emptyset$. Тогда $A'$ есть подразбиение $M'$. 
%\end{example}

Пусть $M$ --- произвольный выпуклый многогранник, $\K'$ --- некоторое подразбиение комплекса $\K(M)$. Заметим, что $\K'$ естественным образом наделяет многогранник $M$ структурой $CW$-комплекса. Открытыми клетками будут относительные внутренности многогранников из $\K'$. Далее мы построим некоторый коцепной комплекс $C_{\K'}^*$ с коэффициентами в поле $\Z_2$, такой, что размерность $H^n(C_{\K'}^*)$ совпадает с $n$-ой компонентой $f$-вектора многогранника $M$.

Пусть $A$ --- произвольная грань многогранника $M$.  Рассмотрим подкомплекс $\A'$ комплекса $\K'$, состоящий из всех многогранников $S\in \K'$, целиком лежащих в $A$. Другими словами $\A'=\{ S \in \K' \mid \mathring S\cap A \neq \emptyset\}$. А также рассмотрим подкомплекс $\partial \A'$ комплекса $\A'$, состоящий из всех многогранников $T\in \A'$, целиком лежащих в $\partial A$. 

Пусть  $C^*(\A', \partial \A'; \Z_2)$ ---  комплекс клеточных коцепей пары $CW$-комплексов, заданной парой полиэдральных комплексов $(\A', \partial \A')$, с коэффициентами в $\Z_2$. Возьмём прямую сумму таких комплексов коцепей по всем граням многогранника $M$:

\begin{definition}\label{def}
\begin{equation}
C_{\K'}^*:= \bigoplus_{A\in \K(M)}C^*(\A', \partial \A'; \Z_2).
\end{equation}
\end{definition} 

Заметим, что группы коцепей комплекса $C_{\K'}^*$ такие же, как у $C^*(\K'; \Z_2)$, но кограничные гомоморфизмы у них разные.
 
\begin{lemma} \label{lem1}
$\dim H^n(C_{\K'}^*)=f_n(M)$.
\end{lemma}
\begin{proof}
Заметим, что в силу определения \ref{def} справедлива формула
\begin{equation}\label{dsumcoh}
H^n(C_{\K'}^*)=\bigoplus_{A\in \K(M)}H^n (\A', \partial \A'; \Z_2).
\end{equation} 
Так как $(\A', \partial \A')$ есть конечная $CW$-пара, то 
\begin{equation}\label{iso}
H^*(\A', \partial \A';\Z_2)\simeq \widetilde{H}^*(\A'/\partial \A';\Z_2),
\end{equation}
где $\A'/\partial \A'$ --- это $CW$-комплекс,  полученный из $\A'$ стягиванием в точку  $|\partial \A'|$. 
Заметим, что  
\begin{equation}\label{sphere}
\widetilde{H}^n(\A'/\partial \A';\Z_2)\simeq \widetilde{H}^n(S^{\dim A};\Z_2)\simeq
\begin{cases}
\Z_2,\ &\text{если}\ n=\dim A,\\
0,\ &\text{иначе}.
\end{cases}
\end{equation}
Из  (\ref{iso}) и (\ref{sphere}) следует, что 
\begin{equation}\label{dim}
\dim\widetilde{H}^n(\A', \partial \A';\Z_2)=
\begin{cases}
1,\ &\text{если}\ n=\dim A,\\
0,\ &\text{иначе}.
\end{cases}
\end{equation}
Из (\ref{dsumcoh}), (\ref{dim}) следует, что  $\dim H^n(C_{\K'}^*)=\#\big\{A\mid A\in \K(M),\ \dim A=n\big\}=f_n(M).$
\end{proof}

\begin{example}\label{firstexample}

Пусть $M$ --- треугольник $V_0 V_1V_2$ (см.\ рис.\ \ref{trsubdiv}а). В качестве  $\K'$ возьмём подразбиение (см.\ рис.\ \ref{trsubdiv}б). Построим в этом случае коцепной комплекс $C_{\K'}^*$ и вычислим вручную его когомологии.

\begin{figure}[h]
\begin{minipage}[h]{0.49\linewidth}
\center{\includegraphics[width=0.45\linewidth]{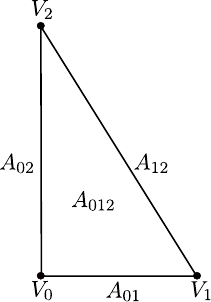} \\ а)}
\end{minipage}
\hfill
\begin{minipage}[h]{0.49\linewidth}
\center{\includegraphics[width=0.45\linewidth]{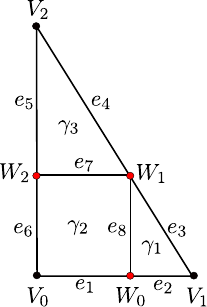} \\ б)}
\end{minipage}
\caption{а) $M$ --- треугольник $V_0V_1V_2$, б)   подразбиение $\K'$ комплекса $\K(M)$.}
\label{trsubdiv}
\end{figure}

На рис. \ref{trsubdiv} для рёбер и двумерной грани треугольника $V_0V_1V_2$, а также для относительных внутренностей многогранников  подразбиения $\K'$ введены специальные обозначения, которые мы далее будем использовать.

Действуя по определению \ref{def}, для каждой грани треугольника $V_0 V_1V_2$ построим комплекс клеточных коцепей соответствующей пары, то есть выпишем группы коцепей, а также кограничные гомоморфизмы. Внесём эти данные построчно в таблицу, чтобы потом взять нужную прямую сумму. Далее для клетки $\omega$ через $\omega^*$ обозначена базисная коцепь, принимающая на $\omega$ значение $1$ и ноль на остальных клетках. Прочерки в ячейках уже упомянутой таблицы соответствуют нулевым кограничным гомоморфизмам.

\[\begin{array}{|c|c|c|c|c|c|}
\hline
\text{грань}&C^0(\A', \partial \A')&C^1(\A', \partial \A')&C^2(\A',\partial\A')&\delta^0&\delta^1\\
\hline
V_0&\langle V_0^*\rangle&0&0&-& -\\
\hline
V_1&\langle V_1^*\rangle&0&0&-& -\\
\hline
V_2&\langle V_2^*\rangle&0&0&-& -\\
\hline
A_{01}&\langle W_0^*\rangle&\langle e_1^*,\ e_2^* \rangle&0&\delta^0(W_0^*)=e_1^*+e_2^*&-\\
\hline
A_{12}&\langle W_1^*\rangle&\langle e_3^*,\ e_4^* \rangle&0&\delta^0(W_1^*)=e_3^*+e_4^*&-\\
\hline
A_{02}&\langle W_2^*\rangle&\langle e_5^*,\ e_6^* \rangle&0&\delta^0(W_2^*)=e_5^*+e_6^*&-\\
\hline
A_{012}&0&\langle e_7^*,\ e_8^* \rangle& \langle \gamma_1^*,\ \gamma_2^*,\ \gamma_3^* \rangle&-&\delta^1(e_7^*)=\gamma_2^*+\gamma_3^*\\
&&&&&\delta^1(e_8^*)=\gamma_1^*+\gamma_2^*\\
\hline
\end{array}\]

Осталось взять прямую сумму комплексов $C^*(\A',\partial\A')$ по всем граням треугольника $V_0V_1V_2$ и увидеть, что $C^0_{\K'}\simeq\langle V_0^*,V_1^*,V_2^*,W_0^*,W_1^*,W_2^*\rangle\simeq C^0(\K';\Z_2)$, 
$C^1_{\K'}\simeq\langle e_1^*,\dots, e_8^*\rangle\simeq C^1(\K';\Z_2)$ и $C^2_{\K'}\simeq\langle \gamma_1^*,\gamma_2^*,\gamma_3^*\rangle\simeq C^2(\K';\Z_2)$  (как и следовало из определения \ref{def}), но кограничные гомоморфизмы ненулевые только на базисных коцепях, двойственных тем клеткам, которые лежат в относительной внутренности некоторой грани треугольника размерности строго большей, чем размерность самой клетки.

\[
H^0(C_{\K'}^*)\simeq \ker \delta^0\simeq \langle V_0^*,V_1^*,V_2^*\rangle \simeq \Z_2^3,
\]
\[
H^1(C_{\K'}^*) \simeq\langle  e_1^*,\dots, e_6^*\rangle/\langle e_1^*+e_2^*,\ e_3^*+e_4^*,\ e_5^*+e_6^*\rangle \simeq \Z_2^3,
\]
\[
H^2(C_{\K'}^*)\simeq\langle \gamma_1^*,\ \gamma_2^*,\ \gamma_3^* \rangle/\langle \gamma_1^*+\gamma_2^*,\ \gamma_2^*+\gamma_3^* \rangle \simeq \Z_2.
\]

Таким образом, размерности групп когомологий коцепного комплекса $C_{\K'}^*$ совпадают с количеством граней соответствующих размерностей треугольника $M$.

\end{example}

\section{Мечено-порядковый многогранник и его кубо-симплициальное разбиение} \label{KPlambda-section}
{\bfseries\scshape Шаг 2.}
Пусть $\P=(P,\preceq)$ --- конечное частично упорядоченное множество (далее --- чум), $P^*$ --- подмножество $P$, содержащее все экстремальные элементы и 
$\P^*=(P^*,\preceq)$ --- заданный им подчум $\P$ (далее --- подчум отмеченных элементов), $\lambda$ --- сохраняющая порядок функция $\P^*\to\R$.
Векторное пространство $\R^P$ функций $P\to\R$ изоморфно координатному пространству $\R^{\#P}$; значение функции $x\in\R^P$ на элементе $p\in P$
отождествляется с $p$-ой координатой $x_p$.

\begin{definition}\label{MOP}
\textit{Мечено-порядковым многогранником} $O(\P,\lambda)$ называется множество точек $x\in \R^P$, таких что $x_p\leq x_q$, если $p\preceq q$, и $x_a=\lambda(a)$, если $a\in P^*$.
\end{definition}

\begin{example}\label{oplambda}
Пусть $P=\{r,p,q,s,t\}$. Отношение частичного порядка $\preceq$ зададим накрывающими соотношениями: $r \precdot p \precdot q \precdot s$ и $p \precdot t$ (рис. \ref{mop1}а). В качестве $P^*$ выберем трёхэлементное подмножество $\{r, s, t\}$. Сохраняющую порядок функцию $\lambda: \P^* \rightarrow \R$ определим равенствами: $\lambda (r)=0$, $\lambda(s)=2$ и $\lambda(t)=1$ (рис. \ref{mop1}б). Выпишем систему неравенств, определяющую многогранник $O(\P, \lambda)$ и построим его проекцию на плоскость $Ox_px_q$ (рис. \ref{mop1}в).

Согласно определению (\ref{MOP}), 
\[
O(\P,\lambda)=
\begin{cases}
0=x_r \le x_q  \le x_s=2,\\
x_p \le x_t= 1.\\
\end{cases}
\]

\end{example}

\begin{figure}[h]
\begin{minipage}{0.31\linewidth}
\center{\includegraphics[width=0.31\linewidth]{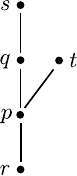} \\ а)}
\end{minipage}
%\hfill
\begin{minipage}{0.31\linewidth}
\center{\includegraphics[width=0.31\linewidth]{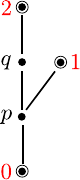} \\ б)}
\end{minipage}
\hfill
\begin{minipage}{0.31\linewidth}
\center{\includegraphics[width=\linewidth]{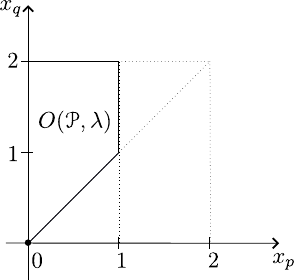} \\ в)}
\end{minipage}
\caption{а) Диаграмма Хассе чума $\P$ из примера \ref{oplambda}; б) диаграмма Хассе чума $\P$ из примера \ref{oplambda}, в которой выделены элементы $P^*$ и указаны значения функции $\lambda$ на них; в) проекция $O(\P,\lambda)$ на координатную плоскость $Ox_px_q$.}
\label{mop1}
\end{figure} 

%Теперь наступает удобный момент, чтобы ввести следующее определение.
\begin{definition}
Полиэдральный комплекс, в котором каждый полиэдр является произведением симплексов, назовём \textit{кубо-симплициальным комплексом}.
\end{definition}

В этом разделе мы  покажем, что мечено-порядковый многогранник $O(\P,\lambda)$ (см. определение \ref{MOP})  обладает некоторым каноническим подразбиением $\K_{\P,\lambda}$, приходящим из структуры чума $\P$. Мы дадим комбинаторное и геометрическое описание подразбиения $\K_{\P,\lambda}$. А также убедимся в том, что $\K_{\P,\lambda}$ является \textit{кубо-симплициальным комплексом}.

Чтобы построить подразбиение $\K_{\P,\lambda}$ многогранника $O(\P,\lambda)$, нам понадобится следующее стандартное определение.

\begin{definition}\label{ideal}
\textit{Порядковым идеалом} чума $\P=(P, \preceq)$ называется подмножество $I\subset P$ такое, что если $x\in I$ и $y\preceq x$ , то $y\in I$.
\end{definition}
\begin{remark}
Порядковые идеалы чума $\P$ также образуют чум по включению. 
\end{remark}

Пусть $\P$, $\P^*$ и $\lambda$ --- это участвующие  в определении $O(\P,\lambda)$ чум, подчум отмеченных элементов и сохраняющая порядок функция из $\P^*$ в $\R$ соответственно. 
\begin{definition}\label{L-lambda}
Цепь $L$ в чуме порядковых идеалов чума $\P$ будем называть \textit{$\lambda$-допустимой} или просто  \textit{допустимой} (когда $\lambda$ ясна из контекста), если она имеет вид $\emptyset=I_0 \subsetneq I_1 \subsetneq  \cdots \subsetneq I_m\subsetneq I_{m+1}=P$ и удовлетворяет условию 
\smallskip

\begin{equation}\label{*}
\begin{aligned}
&\text{ каждое
$\lambda ((I_i \setminus I_{i-1}) \cap P^*)$ пусто или равно $\{t_i\}$ для некоторого $ t_i\in \R$,}\\
&\text{ причём $t_i <  t_j$ при условии, что  $i < j$ и как $t_i$, так и $t_j$ определены.
}
\end{aligned}
\end{equation}
\smallskip

Условие (\ref{*})  означает следующее. Представьте, что мы строим  пирамиду из кубиков --- элементов $P$,  в которой этажи --- это разности между соседними идеалами $I_i$.  Каждый этаж состоит по крайней мере из одного кубика.  Некоторые из наших кубиков являются отмеченными (это элементы $P^*$), на них стоит метка-число (значение функции $\lambda$). Тогда условие (\ref{*}) означает, что отмеченные кубики с одинаковыми метками должны попадать в один и тот же этаж, а также метки добавляемых кубиков должны возрастать от этажа к этажу.
\end{definition}

Пусть теперь $L$ ---  цепь вида $\emptyset=I_0 \subsetneq I_1 \subsetneq  \cdots \subsetneq I_m\subsetneq I_{m+1}=P$  в чуме порядковых идеалов чума $\P$. 
\begin{definition}\label{FL}
Будем обозначать через $F_L=F_L(\lambda)$ множество точек $x\in \R^P$, для которых выполнены три условия:

\begin{equation}\label{condition}
\begin{aligned}
\bullet \hspace{1.7mm} &x|_{P^*}=\lambda,\\
\bullet\hspace{1.7mm} &\text{функция $x:P\rightarrow \R$ постоянна на множествах  $I_1\setminus I_0, \dots, I_{m+1}\setminus I_m$,} \hspace{1.5cm}\\
\bullet\hspace{1.7mm}   &x(I_1) \leq x(I_2\setminus I_1) \leq \cdots \leq x(I_m\setminus I_{m-1}) \leq x(P\setminus I_m).
\end{aligned} 
\end{equation}
\end{definition}

\begin{lemma}\label{LmapstoFL}
Соответствие $L\mapsto F_L$ между допустимыми цепями чума порядковых идеалов чума $\P$ и непустыми множествами вида $F_L$ взаимнооднозначно.
\end{lemma}

\begin{proof}
Сюръективность очевидна. Докажем инъективность. 
Предположим, что различным допустимым цепям $L$ и $L'$ соответствуют одинаковые множества $F_L=F_{L'}$. Это может случиться только в случае, если хотя бы в одной из рассматриваемых цепей (например в $L$)  найдутся разности соседних идеалов $I_k\setminus I_{k-1}$ и $I_{k+1}\setminus I_k$, на которых функции $x\in F_L$ принимают одно и то же постоянное значение, то есть $x(I_k\setminus I_{k-1})=x(I_{k+1}\setminus I_k)=r$. Последнее возможно лишь в случае, когда существуют отмеченные элементы $a,b\in P^*$, такие что $a\in I_k\setminus I_{k-1}$, $b\in I_{k+1}\setminus I_k$ и $\lambda(a)=\lambda(b)=r$. Но это противоречит допустимости цепи $L$ (условию (\ref{*}) нарушено). Поэтому инъективность доказана. 

\end{proof}

\begin{remark}\label{xsavesorder}
Если функция $x$ содержится в $F_L(\lambda)$ для некоторой $\lambda$-допустимой цепи $L$ чума порядковых идеалов чума $\P$, то $x$ --- сохраняющая порядок функция  $\P \to \R$, продолжающая $\lambda$, в частности $x\in O(\P,\lambda)$.
\end{remark}

%\begin{figure}[h]
%\includegraphics[width=0.15\linewidth]{po_set.pdf}
%\includegraphics[width=0.15\linewidth]{markedposet.pdf}
%\includegraphics[width=0.45\linewidth]{FL.pdf}
%\caption{pp}
%\label{posmop0}
%\end{figure} 

\begin{example}\label{ex1}
Пусть $\P$, $\P^*$ и $\lambda$ --- как в примере \ref{oplambda}.

Рассмотрим цепь $L_1=\Big( \emptyset \neq \{r\} \subsetneq  \{r,p\} \subsetneq \{r,p,t\} \subsetneq \{r,p,t,q\} \subsetneq P \Big)$ в чуме порядковых идеалов чума $\P$. 
Из следующей таблицы видно, что она удовлетворяет условию (\ref{*}).
\[\begin{array}{|c|c|c|c|c|c|}
\hline
S   
&\{r\}&\{p\}&\{t\}&\{q\}&\{s\}\\
\hline
\lambda(S\cap P^*)
&\{0\}&\emptyset&\{1\}&\emptyset&\{2\}\\
\hline
\end{array}\]

Соответствующее подмножество $F_{L_1} \subset\R^5$ состоит из таких точек $x$, что $0=x_r \leq x_p \leq x_t=1 \leq x_q \leq x_s=2$. Проекция $F_{L_1}$ на плоскость $Ox_px_q$ является квадратом $0 \leq x_p \leq 1 \leq x_q \leq 2$.

Покажем, что цепь $L=\Big(\{r\} \subsetneq  \{r,p\} \subsetneq \{r,p,q\}\Big)$ не является допустимой. В самом деле, $\lambda((P \setminus\{r,p,q\})  \cap P^*)=\lambda(\{s,t\} \cap P^*)=\{1,2\}$, что противоречит условию (\ref{*}). 

Если мы рассмотрим соответствующее подмножество $F_L\subset \R^5$, то увидим, что оно окажется пустым,  так как по определению должно состоять из  таких точек $x$, что $0=x_r \leq x_p \leq x_q \leq x_t=1=2=x_s$.
\end{example}

\begin{proposition}\label{KP}
Совокупность \[\K_{\P,\lambda}:=\{ F_L\mid L\text{ --- $\lambda$-допустимая цепь в чуме порядковых идеалов }\P\}\] является полиэдральным комплексом.
\end{proposition}

\begin{proof}
Покажем, что для $\lambda$-допустимой цепи $L$ в чуме порядковых идеалов чума $\P$ условия (\ref{condition}) задают систему непротиворечивых линейных неравенств и равенств в пространстве $\R^P$ и тем самым определяют непустой выпуклый полиэдр.

Пусть $P$ состоит из элементов $p_1,\ldots, p_l$.
Всякой сохраняющей порядок биекции $\sigma:\P \rightarrow \{1, 2,\ldots,l\}$ соответствует линейное расширение $\P_\sigma=(P,\preceq_\sigma)$ чума $\P$; а именно, $s\preceq_\sigma t$, если $\sigma(s)\le\sigma(t)$. 
Легко видеть, что это соответствие взаимно однозначно.
Биекция $\sigma$ задаёт перестановку $q_1,\ldots, q_l$ элементов 
$p_1,\ldots,p_l$, где $q_i=\sigma^{-1}(i)$. 

\begin{lemma}
Между множеством линейных расширений чума $\P$ и множеством  цепей максимальной длины (то есть длины  $l+1$) в чуме порядковых идеалов чума $\P$ существует взаимно однозначное соответствие.
\end{lemma}
\begin{proof}

Действительно, пусть $\P_\sigma$ --- линейное расширение чума $\P$. Поставим ему в соответствие  цепь $L_{\sigma}=\big(I_0 \subsetneq I_1 \subsetneq \ldots \subsetneq I_l\big)$, 
где  $I_{i}=\sigma^{-1}\big(\{1,2,\ldots,i\}\big)$. Обратно, пусть дана  цепь максимальной длины $L=\big(\emptyset= I_0 \subsetneq I_1 \subsetneq \ldots \subsetneq I_l=P\big)$.
Тогда каждая разность $I_{k+1} \setminus I_k$ есть одноэлементное множество $\{u\}$, и поскольку $I_k$ --- порядковый идеал, для любого элемента $v \in I_k$ либо $v \prec u$, либо $u$ и $v$ несравнимы. 
Поставим в соответствие цепи $L$ линейное расширение $\P_\tau$ чума $\P$ такое, что каждое $\tau(I_i\setminus I_{i-1})=\{i\}$.
\end{proof}

Сформулируем без доказательства достаточно очевидную лемму.
\begin{lemma}
Цепь $L_{\sigma}$ не является $\lambda$-допустимой в точности в следующих двух случаях. Случай 1: функция $\lambda$ не сохраняет линейный порядок $\preceq_\sigma$, то есть существует пара 
$\preceq$-несравнимых элементов $a,b \in P^*$ таких, что $a \preceq_\sigma b$, но $\lambda(a) > \lambda(b)$. Случай 2:  $\lambda$ сохраняет $\preceq_\sigma$, но существуют различные элементы из $P^*$, на которых $\lambda$ принимает одинаковые значения. 
\end{lemma}

Заметим, что в Случае 2 цепь $L_\sigma$ можно проредить, не меняя множества $F_{L_\sigma}$, так, чтобы она стала $\lambda$-допустимой. Действительно, предположим, что $q_k$ и $q_{k+n}$ из  $P^*$ таковы, что $\lambda(q_k)=\lambda(q_{k+n})$; тогда удалим из цепи $L_{\sigma}$ идеалы с $I_{k}$ по $I_{k+n-1}$. Повторим эту процедуру для каждой пары различных элементов из $P^*$, на которых $\lambda$ принимает одинаковые значения. В результате получим $\lambda$-допустимую цепь $\widetilde{L}_{\sigma}$.

Если цепь $L_{\sigma}$ является $\lambda$-допустимой, положим $\widetilde{L}_{\sigma}=L_\sigma$.
Таким образом, цепь $\widetilde{L}_{\sigma}$ определена для всех \textit{$\lambda$-согласованных} линейных расширений $\P_\sigma$ чума $\P$, то есть таких, что $\lambda$ сохраняет $\preceq_\sigma$.

Пусть $\P_\sigma$ --- $\lambda$-согласованное линейное расширение $\P$ и $q_i=\sigma^{-1}(i)$.
Тогда множество точек $F_{\widetilde{L}_{\sigma}}$ является многогранником в $\R^P$, заданным системой равенств и неравенств, получающейся из  системы $x_{q_1} \leq x_{q_2} \leq \ldots \leq x_{q_l}$ заменой некоторых знаков неравенства на знаки равенства и всех координат $x_{q_i}$ при $q_i \in P^*$ на числа $\lambda(q_i)$  в соответствии с условиями (\ref{condition}). 
Ограниченность $F_{\widetilde{L}_{\sigma}}$ следует из условия, что $P^*$ содержит все экстремальные элементы чума $\P$.

Итак, по каждому $\lambda$-согласованному линейному расширению $\P_\sigma$ чума $\P$ мы построили $\lambda$-допустимую цепь $\widetilde{L}_{\sigma}$, определяющую многогранник 
$F_{\widetilde{L}_{\sigma}}\in\K_{\P,\lambda}$. Любая другая $\lambda$-допустимая цепь $L$ в чуме порядковых идеалов чума $\P$ будет состоять из меньшего числа звеньев и может быть получена из $\widetilde{L}_{\sigma}$ для некоторого (возможно не единственного) линейного расширения $\P_\sigma$ чума $\P$ удалением некоторого набора идеалов. Поэтому соответствующее ей множество точек $F_L$ является многогранником, заданным системой неравенств и равенств, получающихся из системы, определяющей многогранник $F_{\widetilde{L}_{\sigma}}$, заменой некоторых знаков неравенства на знаки равенства в соответствии с условиями (\ref{condition}). Соблюдение условия (\ref{*}) гарантирует непротиворечивость полученной системы. А это, в частности,  означает, что $F_L$ является непустой гранью многогранника $F_{\widetilde{L}_{\sigma}}$. Таким образом, 
элементы $\K_{\P,\lambda}$ максимальной размерности всегда имеют вид $F_{\widetilde{L}_{\sigma}}$ для некоторого линейного расширения $\P_\sigma$ чума $\P$ (сравн.\ пример \ref{Fsigma}).

Покажем далее, что все грани многогранника $F_{\widetilde{L}_{\sigma}}$ принадлежат $\K_{\P,\lambda}$. Пусть $G$ --- непустая грань многогранника $F_{\widetilde{L}_{\sigma}}$. Тогда определяющая система неравенств и равенств для $G$ --- назовём её $S_G$ --- получается из системы, определяющей многогранник $F_{\widetilde{L}_{\sigma}}$, не приводящей к противоречию заменой некоторых знаков неравенства на знаки равенства. Заметим, что, как и раньше, система $S_G$  получается из  системы $x_{q_1} \leq x_{q_2} \leq \ldots \leq x_{q_l}$ заменой некоторых знаков неравенства на знаки равенства и всех координат $x_{q_i}$ при $q_i \in P^*$ на числа $\lambda(q_i)$. Выпишем цепь $L$, такую, что $F_L=G$,  следующим очевидным образом: в $I_1$ помещаем все элементы $P$, на которых функции $x$ из $G$ принимают значения, стоящие в системе $S_G$ до первого знака неравенства (движемся слева направо); в $I_2$ помещаем все элементы $P$, на которых функции $x$ из $G$ принимают значения, стоящие до второго знака неравенства и так далее. Построенная цепь будет удовлетворять условию (\ref{*}), иначе соответствующая система была бы противоречивой. 

Для завершения доказательства остаётся доказать следующую лемму.

\begin{lemma}\label{FLcapFL'}
В случае, если пересечение $\lambda$-допустимых цепей $L$ и $L'$ в чуме порядковых идеалов чума $\P$ является $\lambda$-допустимой цепью, то $F_{L}\cap F_{L'}\in \K_{\P,\lambda}$, а именно, $F_{L}\cap F_{L'} = F_{L\cap L'}$, причём $F_{L\cap L'}$  --- грань каждого из многогранников $F_L$ и $F_{L'}$;  а в случае, если пересечение $L\cap L'$  не является $\lambda$-допустимой цепью, то пересечение $F_L$ и $F_{L'}$ пусто.
\end{lemma}

\begin{proof}
Пусть пересечение $L\cap L'= \Big(\emptyset =J_0\subsetneq J_1\subsetneq\cdots\subsetneq J_k\subsetneq J_{k+1}=P\Big)$ --- $\lambda$-допустимая цепь. Заметим, что в каждой из цепей $L$ и $L'$ разность $J_r \setminus J_{r-1}$ (где $1\leq r\leq k+1$) подразбита своей цепочкой вложенных идеалов (см.\ рис.\ \ref{LcapL'}).
\begin{figure}[h]
\includegraphics[width=0.75\linewidth]{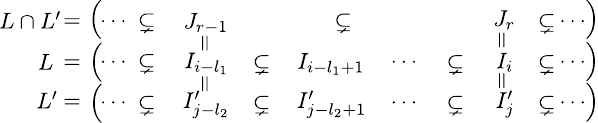}
\caption{Связь между цепями вложенных идеалов и их пересечением.}
\label{LcapL'}
\end{figure}

По определению \ref{FL} справедливы включения $F_{L\cap L'} \subset F_L$ и $F_{L\cap L'} \subset F_{L'}$, значит $F_{L\cap L'} \subset F_L \cap F_{L'}$ (это включение, в частности, гарантирует нам непустоту $F_L\cap F_{L'}$). Теперь докажем, что $F_{L}\cap F_{L'} \subset F_{L\cap L'}$. Пусть  функция  $x\in F_L \cap F_{L'}$, покажем, что $x\in F_{L\cap L'}$. Согласно (\ref{condition}) последнее равносильно условиям а) $x|_{P^*}=\lambda$, б) $x$ постоянна на множествах $J_1\setminus J_0, \dots, J_{k+1}\setminus J_k$ и в) $x(J_1) \leq x(J_2\setminus J_1) \leq \cdots \leq x(J_k\setminus I_{k-1}) \leq x(P\setminus J_k)$. Заметим, что условия а) и в) выполнены очевидным образом, так что остаётся проверить справедливость условия б). Предположим, что $x$ не постоянна на множествах $J_1\setminus J_0, \dots, J_{k+1}\setminus J_k$, то есть найдётся  разность соседних идеалов $J_r \setminus J_{r-1}$ цепи $L\cap L'$ такая, что функция $x$ принимает на ней по крайней мере два различных значения. Допустим, значения функции $x$ различны на элементах $p,q\in J_r \setminus J_{r-1}$, то есть $x_p=r_1\neq r_2=x_q$. Без потери общности можно предположить, что $r_1<r_2$. Так как $x\in F_L \cap F_{L'}$, то согласно замечанию \ref{xsavesorder}, $x$ --- сохраняющая порядок функция $\P \to \R$.
Заметим, что порядковые идеалы чума $\R$ --- это лучи $(-\infty, r)$, где $r$ --- некоторое действительное число. Рассмотрим порядковый идеал $(-\infty; r_1)$ чума $\R$, назовём его $I$. Полный прообраз этого идеала $x^{-1}(I)$ является идеалом чума $\P$, назовём его $J$ (полный прообраз порядкового идеала при монотонном отображении чумов является порядковым идеалом). По определению \ref{FL} идеал $J$ содержится как в цепи $L$, так и в цепи $L'$, а значит принадлежит их пересечению. Заметим, что $J_{r-1} \subsetneq J\subsetneq J_r$. Однако мы предполагали, что идеалы $J_{r-1}$ и $J_r$ в цепи $L\cap L'$ --- соседние. Это противоречие доказывает, что функция $x$ из $F_L\cap F_{L'}$ постоянна на множествах $J_1\setminus J_0, \dots, J_{k+1}\setminus J_k$. Мы показали, что любая функция $x$ из  $F_L\cap F_{L'}$ принадлежит также и $F_{L\cap L'}$, то есть $F_L\cap F_{L'} \subset F_{L\cap L'}$. Таким образом, если $L\cap L'$ --- $\lambda$-допустимая цепь, то  справедливо равенство $F_L\cap F_{L'}=F_{L\cap L'}$. Несложно понять, что $F_{L\cap L'}$ является гранью каждого из многогранников $F_L$ и $F_{L'}$.

Теперь пусть пересечение $L\cap L'= \Big(\emptyset =J_0\subsetneq J_1\subsetneq\cdots\subsetneq J_k\subsetneq J_{k+1}=P\Big)$ НЕ является $\lambda$-допустимой цепью. Цепи $L$ и $L'$ --- $\lambda$-допустимые, поэтому достаточно очевидно, что $L\cap L'$  может не быть $\lambda$-допустимой только в том случае, если некоторая разность соседних идеалов $J_r\setminus J_{r-1}$ цепи $L\cap L'$ содержит элементы $a,b \in P^*$, такие что $\lambda(a) < \lambda(b)$. Предположим, что $F_L\cap F_{L'}$ не пусто, то есть найдётся функция $x$, которая принадлежит как $F_L$, так и $F_{L'}$. Вновь согласно замечанию \ref{xsavesorder}, $x$ --- сохраняющая порядок функция из $\P$ в $\R$, продолжающая $\lambda$. Поэтому $x_a=\lambda(a)$ и $x_b=\lambda(b)$. В чуме $\R$ рассмотрим порядковый идеал $(-\infty;\lambda(a))$, назовём его $I'$. Как мы уже знаем, полный прообраз этого идеала $x^{-1}(I')$ является идеалом чума $\P$ общим для цепей $L$ и $L'$ (то есть $I'\in L\cap L'$). Назовём его $J'$. Идеал $J'$ по построению содержит идеал $J_{r-1}$, содержится в идеале $J_r$ и не совпадает ни с одним из них. Однако мы предполагали, что идеалы $J_r$ и $J_{r-1}$ в цепи $L\cap L'$ --- соседние. Получили противоречие. Поэтому если цепь $L\cap L'$  не является $\lambda$-допустимой цепью, то пересечение $F_L\cap F_{L'}$ пусто.
\end{proof}

На этом доказательство Предложения завершено.
\end{proof}

\begin{example}\label{Fsigma}
Рассмотрим чум из примера \ref{ex1}. Найдём многогранники максимальной размерности комлекса $\K_{\P,\lambda}$. Из доказательства предложения \ref{KP} мы знаем, что они всегда имеют вид $F_{\widetilde{L}_{\sigma}}$ для некоторого линейного расширения $\P_\sigma$ чума $\P$. Так как в нашем случаем в чуме $\P$ всего два несравнимых элемента, то и линейных расширений будет в точности два: $\P_{\sigma_1}$ и $\P_{\sigma_2}$ ( рис.\ \ref{sigma}а и б). Многогранник $F_{L_{\sigma_1}} = F_{L_1}$ из примера \ref{ex1}. Многогранник $F_{L_{\sigma_2}}\subset  \R^5$ состоит из таких точек $x$, что $0=x_r \leq x_p \leq x_q \leq x_t=1 \leq x_s =2$. Проекция $F_{L_{\sigma_2}}$ на $Ox_px_q$ является треугольником (рис.\ \ref{sigma}в).
\end{example}

\begin{figure}[h]
\begin{minipage}[h]{0.31\linewidth}
\center{\includegraphics[width=0.12\linewidth]{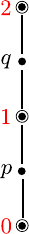} \\ а)}
\end{minipage}
%\hfill
\begin{minipage}[h]{0.31\linewidth}
\center{\includegraphics[width=0.12\linewidth]{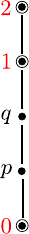} \\ б)}
\end{minipage}
\hfill
\begin{minipage}[h]{0.31\linewidth}
\center{\includegraphics[width=\linewidth]{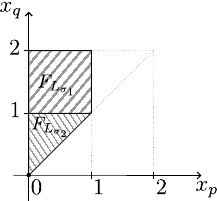} \\ в)}
\end{minipage}
\caption{а) Диаграмма Хассе $\P_{\sigma_1}$; б) диаграмма Хассе $\P_{\sigma_2}$; в) проекция подразбиения $\K_{\P,\lambda}$ многогранника $O(\P,\lambda)$ на плоскость $Ox_px_q$.}
\label{sigma}
\end{figure}

\begin{corollary}\label{subchain}
Многогранник $F_L$ является гранью $F_{L'}$ если и только если $\lambda$-допустимая цепь $L$ содержится в $\lambda$-допустимой цепи $L'$ (то есть цепь $L$ получается из $L'$ удалением некоторого набора идеалов).
\end{corollary}

\begin{proof}
Пусть $F_L$ является гранью $F_{L'}$. Тогда $F_L\cap F_{L'}=F_L$. Таким образом пересечение $F_L\cap F_{L'}$ не пусто. Тогда из леммы \ref{FLcapFL'} следует, что $L\cap L'$ является $\lambda$-допустимой цепью. Применяя снова лемму \ref{FLcapFL'}, получаем, что $F_L\cap F_{L'}=F_{L\cap L'}$. Таким образом, $F_L=F_{L\cap L'}$. Но соответствие $L\mapsto F_L$ является инъективным (см.\ лемму \ref{LmapstoFL}), поэтому $L=L\cap L'$, то есть $L$ содержится в $L'$.

Теперь пусть $L\subset L'$, тогда $L\cap L'=L$. Так как каждая из цепей $L$ и $L'$ по условию является $\lambda$-допустимой, то и $L\cap L'$  является $\lambda$-допустимой цепью. В лемме \ref{FLcapFL'} мы показали, что в этом случае $F_L\cap F_{L'}=F_{L\cap L'}$. Так как $L\cap L'=L$, имеем $F_L\cap F_{L'}=F_L$. Поэтому $F_L$ является гранью $F_{L'}$.
\end{proof}

\begin{corollary}\label{KP_OPlambda}
Полиэдральный комплекс $\K_{\P,\lambda}$ является  подразбиением многогранника $O(\P,\lambda)$.
\end{corollary}
\begin{proof}
Нам нужно показать, что  $|\K_{\P,\lambda}|=O(\P,\lambda)$. 

Пусть $\P_\sigma$ --- $\lambda$-согласованное линейное расширение чума $\P$, и пусть $L_\sigma$ --- соответствующая ему цепь максимальной длины в чуме порядковых идеалов чума $\P$ (см.\ начало доказательства предложения \ref{KP}).
Заметим, что многогранник $F_{\widetilde L_{\sigma}} \in \K_{\P,\lambda}$  (см.\ доказательство предложения \ref{KP}) по построению является мечено-порядковым многогранником $O(\P_\sigma, \lambda)$. Из определений линейного расширения чума и мечено-порядкового многогранника (см. \ref{MOP}) следует, что многогранник $O(\P_\sigma, \lambda)$ высекается из многогранника $O(\P,\lambda)$ набором полупространств, поэтому имеем $\bigcup_{\sigma}O(\P_\sigma, \lambda)\subset O(\P,\lambda)$ (где объединение берётся по всем $\lambda$-согласованным линейным расширениям). Покажем, что верно и обратное включение. Пусть функция $x \in O(\P, \lambda)$. Значения функции $x$ на элементах $P$ можно упорядочить по возрастанию, а затем записать связывающую их цепочку неравенств $x_{q_1} \leq x_{q_2} \leq \dots \leq x_{q_{\# P}}$ (такая цепочка может быть не единственна, если функция $x$ принимает одинаковые значения на нескольких элементах $P$). Так как $x \in O(\P,\lambda)$, то перестановке $q_1,\dots, q_{\#P}$ элементов $P$  соответствует некоторое $\lambda$-согласованное линейное расширение $\P_\sigma$.  Мы заключаем, что $x\in {O(\P_{\sigma},\lambda)}$. 

Для любого полиэдрального комплекса верно, что его тело является объединением полиэдров комплекса максимальной размерности. Поэтому $|\K_{\P,\lambda}|= \bigcup_\sigma F_{\widetilde L_{\sigma}}$. Таким образом, $|\K_{\P,\lambda}|= \bigcup_{\sigma} F_{\widetilde L_{\sigma}}=\bigcup_{\sigma}O(\P_\sigma, \lambda)= O(\P,\lambda)$.

\end{proof} 

\begin{remark}
Полиэдральный комплекс $\K_{\P,\lambda}$ является кубо-симплициальным подразбиением многогранника $O(\P,\lambda)$, так как каждый из многогранников $F_{\widetilde L_{\sigma}}=O(\P_{\sigma},\lambda)$ является произведением симплексов.
\end{remark}

\begin{remark}\label{geometricKPlambda}
Подразбиение $\K_{\P,\lambda}$ мечено-порядкового многогранника можно описать геометрически \cite{LMS}*{proof of Theorem 3.4}. А именно, исходя из всевозможных пар несравнимых элементов $\P$, будем рассекать многогранник $O(\P,\lambda)$ некоторым набором гиперплоскостей. Возможны такие варианты: а) $u,v \in P \setminus P^*$ --- несравнимые элементы, ни один из которых не является отмеченным, тогда проведём гиперплоскость $x_u=x_v$, обозначим её через $H_{uv}$; б) $s \in P\setminus P^*$ и $a \in P^*$ --- несравнимые элементы, один из которых является отмеченным, тогда проведём гиперплоскость $x_s=\lambda(a)$, обозначим её через $H_{as}$; в) $b,c \in P^*$ --- несравнимые элементы, которые оба оказались отмеченными, тогда они не участвуют в дополнительных построениях. В результате мы получим некоторое подразбиение $\K'$ многогранника $O(\P,\lambda)$.

Покажем, что каждый $F_{\widetilde L_{\sigma}}=O(\P_\sigma, \lambda)$ принадлежит $\K'$. Напомним, что в  $\P_\sigma$ зафиксирован порядок между элементами, которые были несравнимы в $\P$. Поэтому каждый многогранник $O(\P_\sigma, \lambda)=O(\P,\lambda)\cap \big(\bigcap_{u,v} H_{uv}^{\gtrless}\big)\cap \big(\bigcap_{a,s}H_{as}^{\gtrless} \big)$, где символ $H_{uv}^{\gtrless}$ означает одно из двух полупространств $x_u \leq x_v$ или $x_u \geq x_v$ в зависимости от установленного порядка между элементами $u$ и $v$ в $\P_\sigma$. Таким образом, $O(\P_\sigma, \lambda) \in \K'$. Принимая во внимание предложение \ref{KP} и следствие \ref{KP_OPlambda}, мы можем заключить, что $\K_{\P,\lambda}=\K'$. 
 
\end{remark}

\begin{remark}
Также в работе  \cite{LMS}*{см.\ \S5.2} для случая мечено-порядковых многогранников, построенных по строго планарному чуму с отмеченными элементами \cite{LMS}*{см.\ опред.\ 4.2}, приведена конструкция, которая позволяет, поэтапно преобразуя исходный чум, получить цепи с отмеченными элементами, мечено-порядковые многогранники которых есть многогранники $F_{\widetilde{L}_{\sigma}}$.
\end{remark}

\section{Вычисление $f$-вектора многогранника $O(\P,\lambda)$}\label{fvector-section}
 
Пусть снова $\P$, $\P^*$ и $\lambda$ --- это участвующие  в определении \ref{MOP} чум, подчум отмеченных элементов и сохраняющая порядок функция из $\P^*$ в $\R$ соответственно. 
В этом разделе мы хотим описать коцепной комплекс $C_{\K_{\P,\lambda}}^*$ в чисто комбинаторных терминах, т.е.\ напрямую в терминах чума $\P$ и сохраняющего порядок 
отображения $\lambda$, без ссылок на какие-либо многогранники. 

Для того, чтобы описать в комбинаторных терминах группы коцепей комплекса $C_{\K_{\P,\lambda}}^*$, сначала опишем в комбинаторных терминах размерность многогранника $F_L$ 
(см.\ определение \ref{FL}) для всякой $\lambda$-допустимой цепи $L$ (см.\ определение \ref{L-lambda}).

\begin{definition}\label{defdimL}
Назовём \textit{размерностью} $\lambda$-допустимой цепи $L=\Big(\emptyset=I_0 \subsetneq I_1 \subsetneq  \cdots \subsetneq I_m\subsetneq I_{m+1}=P\Big)$ количество таких разностей соседних идеалов $I_i\setminus I_{i-1}$, которые не пересекаются с $P^*$. 
Обозначение: $\dim L$.  
\end{definition}

Легко видеть, что $\dim F_L= \dim L$.

Заметим, что  $\dim L+3$ для $\lambda$-допустимой цепи $L$ не превосходит \textit{длины} $L$, то есть количества входящих в $L$ идеалов; равенство достигается, например, если $P^*$ содержит только наименьший и наибольший элементы $\P$ ( в этом случае количество таких разностей соседних идеалов $I_i\setminus I_{i-1}$, которые не пересекаются с $P^*$, равно $m-1$, что в точности на $3$ меньше длины цепи $L$).

Несложно описать в комбинаторных терминах и размерность всего многогранника $O(\P,\lambda)$.

\begin{definition}\label{dimPlambda}
Назовём \textit{размерностью} пары $(\P,\lambda)$ максимум размерностей $\lambda$-допустимых цепей в чуме порядковых идеалов чума $\P$.
Обозначение: $\dim (\P,\lambda)$.  
\end{definition}

Легко видеть, что $\dim (\P,\lambda)= \dim O(\P,\lambda)$.

Для того, чтобы описать в комбинаторных терминах коцепной гомоморфизм комплекса $C_{\K_{\P,\lambda}}^*$, нам нужно сначала научиться для произвольной $\lambda$-допустимой цепи $L$ выписывать все содержащие её $\lambda$-допустимые цепи на единицу большей размерности.

\begin{definition}\label{Lk}
Скажем, что $\lambda$-допустимая цепь $L_k$ получена $I_k$-\textit{уплотнением} из $\lambda$-допустимой цепи $L=\Big(\emptyset=I_0 \subsetneq I_1 \subsetneq  \cdots \subsetneq I_m\subsetneq I_{m+1}=P\Big)$, если $L_k$ получена из $L$ добавлением единственного идеала $J_k$, такого что   
$I_{k-1}\subsetneq J_k\subsetneq I_k$ (см.\ рис.\ \ref{LcapLk}). 
\end{definition}

\begin{figure}[h!]
\includegraphics[width=0.85\linewidth]{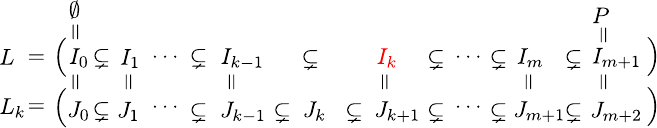}
\caption{$I_k$-Уплотнение $\lambda$-допустимой цепи $L$}
\label{LcapLk}
\end{figure}

Покажем, что всякое $I_k$-уплотнение увеличивает размерность $\lambda$-допустимой цепи ровно на единицу.

\begin{lemma}\label{dimL}
$\dim L_k=\dim L+1$
\end{lemma}

\begin{proof}
 Будем использовать обозначения рис.\ \ref{LcapLk}. Заметим, что $I_k\setminus I_{k-1}= (J_k\setminus J_{k-1})\cup(J_{k+1}\setminus J_k)$. Все остальные разности соседних идеалов у цепей $L$ и $L_k$ совпадают. Возможны два случая. Случай А: $I_{k}\setminus I_{k-1}$ содержит отмеченные элементы $a_1,\dots,a_n$ (все, на которых $\lambda$ принимает одно и то же значение; $n\geq 1$). Так как $L_k$ является по определению $\lambda$-допустимой цепью, то по условию (\ref{*}) либо $J_k\setminus J_{k-1}$ содержит все отмеченные элементы $a_1,\dots,a_n$
 (тогда $J_{k+1}\setminus J_k$ не содержит отмеченных элементов), либо $J_{k+1}\setminus J_k$ содержит все отмеченные элементы $a_1,\dots,a_n$ (тогда $J_k\setminus J_{k-1}$ не содержит отмеченных элементов). Поэтому в цепи $L_k$ ровно на одну больше разностей соседних идеалов, не пересекающихся с $P^*$.
 Случай Б: $I_{k}\setminus I_{k-1}$ не содержит отмеченные элементы. Тогда очевидно $J_k\setminus J_{k-1}$ и $J_{k+1}\setminus J_k$ также не содержат отмеченных элементов. И размерность цепи $L_k$ снова на единицу больше размерности $L$.
\end{proof}
\begin{corollary}\label{dimL+1}
Пусть $L$ и $L'$ --- $\lambda$-допустимые цепи, $L\subset L'$ и $\dim L'=\dim L+1$. Тогда  $L'$ получена из $L$ некоторым $I_k$-уплотнением. 
\end{corollary}
\begin{proof}
Всякое $I_k$-уплотнение $\lambda$-допустимой цепи --- это  минимально возможное увеличение цепи (ровно на один идеал). Предположим, что  $L\subset L'$ и они отличаются на $r$ идеалов. Несложно понять, что в этом случае существует  последовательность вложенных цепей $L\subsetneq L^1\subsetneq \dots \subsetneq L^{r-1} \subsetneq L'$, такая что каждая следующая цепь получается из предыдущей некоторым $I_k$-уплотнением. Тогда, согласно лемме \ref{dimL}, $\dim L'=\dim L +r$. По условию $r=1$, поэтому $L'$ получена из $L$ некоторым $I_k$-уплотнением.
\end{proof}

\begin{remark}
В общем случае в $\lambda$-допустимой цепи $L$ могут встретиться как идеалы $I_k$, такие что $L$ нельзя $I_k$-уплотнить, так и идеалы, такие что $L$ можно $I_k$-уплотнить более чем одним способом. 
\end{remark}

\begin{definition}\label{conjugatechain}
Будем говорить, что $\lambda$-допустимые цепи $L'$ и $L''$, $I_k$-\textit{сопряжены}, если они получены $I_k$-уплотнением из некоторой $\lambda$-допустимой цепи $L$ и справедливы равенства $J_{k+1}\setminus J_k=J_k'\setminus J_{k-1}'$ и $J_k\setminus J_{k-1}=J_{k+1}'\setminus J_k'$.%
\footnote{Другими словами, разность идеалов $I_k\setminus I_{k-1}$ распадается на две части ($I_k\setminus I_{k-1}=A\cup B$), которые, присоединяясь в разном порядке, приводят к образованию $L'$ и $L''$. Имеем $J_k=I_{k-1}\cup A=I_k\setminus B$ и $J_k'=I_{k-1}\cup B=I_k\setminus A$.}
\end{definition}

Пусть $C^i(\P,\lambda)$ --- векторное пространство над $\Z_2$ с базисом, состоящим из всех $\lambda$-допустимых цепей размерности $i$ 
в чуме порядковых идеалов чума $\P$.
Для $\lambda$-допустимой цепи $L=\Big(\emptyset=I_0 \subsetneq I_1 \subsetneq  \cdots \subsetneq I_m=P\Big)$ положим 
\begin{equation}\label{delta}
\delta(L)=\Sigma_1(L)+\dots+\Sigma_m(L),
\end{equation} 
где $\Sigma_k(L)$ --- формальная сумма всех $I_k$-уплотнений $L$, имеющих $I_k$-сопряжённое. Заметим, что $\delta(L)\in C^{i+1}(\P,\lambda)$, так как всякое $I_k$-уплотнение увеличивает размерность $\lambda$-допустимой цепи на единицу (см.\ лемма \ref{dimL}). Продолжив $\delta$ по линейности, получим линейные отображения $\delta:C^i(\P,\lambda)\to C^{i+1}(\P,\lambda)$.
Положим \[C^*(\P,\lambda)=\Big(C^0(\P,\lambda)\xrightarrow{\delta}\dots\xrightarrow{\delta}C^n(\P,\lambda)\Big), \text{ где } n=\dim(\P,\lambda).\]

\begin{theorem}\label{main}
$C^*(\P,\lambda)$ является коцепным комплексом (т.е.\ $\delta\circ\delta=0$) и изоморфен комплексу $C^*_{\K_{\P,\lambda}}$.
\end{theorem}
\begin{proof}
Для начала построим изоморфизм градуированных векторных пространств $\bigoplus_{i=0}^{\dim (\P,\lambda)} C^i(\P,\lambda)$ и $\bigoplus_{i=0}^{\dim O(\P,\lambda)} C^i_{\K_{\P,\lambda}}$ (напомним, что размерность пары $(\P,\lambda)$ совпадает с размерностью многогранника $O(\P,\lambda)$, см.\ определение \ref{dimPlambda}). Как уже упоминалось после определения \ref{def}, группы коцепей комплекса $C^*_{\K_{\P,\lambda}}$ такие же, как у $C^*(\K_{\P,\lambda};\Z_2)$ (действительно, каждая клетка $CW$-комплекса, соответствующего подразбиению выпуклого многогранника, содержится в относительной внутренности единственной грани этого многогранника). Поэтому $C^i_{\K_{\P,\lambda}}$ --- векторное пространство над $\Z_2$ с базисом, состоящим из коцепей, принимающих значение $1$ на внутренностях $\mathring F_L$ всех  $i$-мерных многогранников вида $ F_L$ и ноль на остальных клетках. Далее в нашем рассуждении мы будем отождествлять базисные коцепи  $C^i_{\K_{\P,\lambda}}$ с соответствующими многогранниками $F_L$, если это не отражается на строгости изложения.

Мы знаем, что соответствие $\phi : L\mapsto F_L$ между $\lambda$-допустимыми цепями в чуме порядковых идеалов чума $\P$ и непустыми многогранниками вида $F_L$ является взаимнооднозначным (см.\ лемма \ref{LmapstoFL}), причём $\dim L=\dim F_L$ (см.\ определение \ref{defdimL}). Продолжив $\phi$ по линейности для каждого $i$, получим градуированный изоморфизм 
\[\phi : \bigoplus_{i=0}^{\dim (\P,\lambda)} C^i(\P,\lambda)\to \bigoplus_{i=0}^{\dim O(\P,\lambda)} C^i_{\K_{\P,\lambda}}.\]

Теперь рассмотрим диаграмму:
\begin{equation}
\begin{CD}\label{deltaphi}
C^i(\P,\lambda) @>\delta>> C^{i+1}(\P,\lambda) \\
@V\phi VV@VV\phi V\\
C^i_{\K_{\P,\lambda}} @>d>> C^{i+1}_{\K_{\P,\lambda}}
\end{CD},
\end{equation}
где через $d$ обозначен коцепной гомоморфизм комплекса $C^*_{\K_{\P,\lambda}}$.

Доказав коммутативность (\ref{deltaphi}), то есть справедливость равенства $\phi\circ\delta=d\circ\phi$, мы получим сразу и $\delta\circ\delta=0$, и изоморфизм рассматриваемых коцепных комплексов.

Опишем композицию $\phi\circ\delta$. Пусть $L$ --- $\lambda$-допустимая цепь размерности $i$. Мы определили $\delta(L)=\Sigma_1(L)+\dots+\Sigma_m(L)$, где $\Sigma_k(L)$ --- сумма всех $I_k$-уплотнений $L$, имеющих $I_k$-сопряжённое.  Заметим, что 
$\Sigma_k(L)$ разбивается на пары $I_k$-сопряжённых друг другу цепей: $\Sigma_k(L)=(L_k^1+(L_k^1)')+\dots +(L_k^{l_k}+(L_k^{l_k})')$, где $L_k^j$ и $(L_k^j)'$ являются $I_k$-сопряженными. Поэтому $\delta(L)$ можно переписать в виде: \[\delta(L)=\sum_{k=1}^{m} \big((L_k^1+(L_k^1)')+\dots +(L_k^{l_k}+(L_k^{l_k})')\big).\]
Таким образом, 
\begin{equation}\label{phidelta}
\phi(\delta(L))=\sum_{k=1}^{m} \big((F_{L_k^1}+F_{(L_k^1)'})+\dots +(F_{L_k^{l_k}}+F_{(L_k^{l_k})'})\big).
\end{equation} 
Заметим, что все многогранники из правой части формулы (\ref{phidelta}) имеют размерность $\dim F_L+1$ (так как $\dim F_{L_k^j}=\dim L_k^j$), содержат $F_L$ (см.\ следствие \ref{subchain}) и соответствуют $I_k$-уплотнениям цепи $L$, имеющим $I_k$-сопряжённое для некоторого $1\leq k\leq m$. Более того, все многогранники подразбиения $\K_{\P,\lambda}$ с перечисленными свойствами входят в правую часть формулы (\ref{phidelta}).

Теперь опишем композицию $d\circ\phi$. Пусть снова $L$ --- $\lambda$-допустимая цепь размерности $i$, тогда $\phi(L)=F_L$, где $F_L\in \K_{\P,\lambda}$ и $\dim F_L=i$. Таким образом, $d(\phi(L))=d(F_L)$ и нам остаётся доказать, что $d(F_L)$ совпадает с правой частью формулы (\ref{phidelta}). 

Согласно определению \ref{def}, $d(F_L)$ --- формальная сумма с коэффициентами из $\Z_2$ многогранников $Q\in \K_{\P,\lambda}$, таких что выполнены три условия:
\begin{equation}\label{dFL}
\begin{aligned}
\text{a)}\, &F_L\subset Q\  (\text{другими словами } F_L \text{ является гранью } Q),\\
\text{б)}\, &\dim Q=\dim F_L+1,\hspace{5.5cm}\\
\text{в)}\, &\mathring F_L \text{ и }\mathring Q \text{ лежат в относительной внутренности одной и той же грани } A\\ 
 &\text{многогранника } O(\P,\lambda).
\end{aligned} 
\end{equation}
 
Мы уже знаем из следствий \ref{subchain} и \ref{dimL+1}, что условия (\ref{dFL} а,б) означают, что $Q=F_{L_k}$, где $L_k$ --- $\lambda$-допустимая цепь, которая получена из $L$ некоторым $I_k$-уплотнением. Напомним, что, согласно определению \ref{FL}, $F_L$ состоит из функций $x\in\R^{\P}$, для которых выполнены условия: 

А) $x|_{P^*}=\lambda$, 

Б) $x$ постоянна на множествах $I_1\setminus I_0, \dots, I_{m}\setminus I_{m-1}$ и 

В) $x(I_1) \leq x(I_2\setminus I_1) \leq \cdots \leq x(I_k\setminus I_{k-1})\leq \cdots \leq x(P\setminus I_{m-1})$. 

Согласно определению \ref{Lk}, цепь $L_k$  получена из $L$ добавлением единственного идеала $J_k$, такого что $I_{k-1}\subsetneq J_k\subsetneq I_k$. Соответственно $F_{L_k}$ состоит из функций $x\in\R^{\P}$, для которых выполнены условия: 

А) $x|_{P^*}=\lambda$, 

Б$'$) $x$ постоянна на множествах $I_1\setminus I_0, \dots, J_k\setminus I_{k-1},\  I_k\setminus J_k, \dots, I_{m}\setminus I_{m-1}$ и 

В$'$) $x(I_1) \leq x(I_2\setminus I_1)\leq\cdots \leq x(J_k\setminus I_{k-1})\leq x(I_k\setminus J_k)\leq\cdots \leq x(P\setminus I_{m-1})$. 

Таким образом, система  В$'$) отличается от В) тем, что значения функции $x$ на множествах $J_k\setminus I_{k-1}$ и $I_k\setminus J_k$ не совпадают, а связаны неравенством $x(J_k\setminus I_{k-1})\leq x(I_k\setminus J_k)$. Так как $F_L\subset F_{L_k}$, то и для аффинных оболочек выполнено $\Aff(F_L)\subset \Aff(F_{L_k})$. Заметим, что аффинные оболочки $\Aff(F_L)$ и $\Aff(F_{L_k})$ задаются системами равенств из условий Б) и Б$'$). Пусть $p\in J_k\setminus I_{k-1}$, $q\in I_k\setminus J_k$,  $H_{pq}$ --- гиперплоскость, заданная равенством $x_p=x_q$ и $H_{pq}^{\leq}$ --- полупространство, заданное неравенством $x_p\leq x_q$. Несложно видеть, что $\Aff(F_L)=\Aff(F_{L_k})\cap H_{pq}$ и что  $F_{L_k}\subset\Aff(F_{L_k})\cap H_{pq}^{\leq}$.

Теперь представим, что мы живём в $\mathring F_{L_k}$ и хотим отправиться в путешествие по $\Aff(F_{L_k})$. Согласно геометрическому описанию \ref{geometricKPlambda}, выйдя из $\mathring F_{L_k}$  и пройдя через $\mathring F_L$, мы либо покидаем многогранник $O(\P,\lambda)$, либо попадаем в относительную внутренность некоторой грани $F_{L'_k}$ нашего разбиения, такой что $\Aff F_{L'_k}=\Aff F_{L_k}$. Однако мы знаем, что  $\mathring F_L$ и $\mathring F_{L_k}$ лежат в относительной внутренности одной и той же грани $A$ многогранника $O(\P,\lambda)$ (см.\ условие \ref{dFL} в). Поэтому в рассматриваемом случае, выйдя из $\mathring F_{L_k}$  и пройдя через $\mathring F_L$, мы не покидаем многогранник $O(\P,\lambda)$, а значит  попадаем в относительную внутренность некоторой грани $F_{L'_k}$ нашего разбиения. Несложно показать, что $F_{L'_k}$ состоит из функций $x\in \R^{\P}$, для которых выполнены условия:

А) $x|_{P^*}=\lambda$, 

Б$'$) $x$ постоянна на множествах $I_1\setminus I_0, \dots, J_k\setminus I_{k-1},\  I_k\setminus J_k, \dots, I_{m}\setminus I_{m-1}$ и  

В$''$) $x(I_1) \leq x(I_2\setminus I_1)\leq\cdots \leq x(I_k\setminus J_k)\leq x(J_k\setminus I_{k-1})\leq\cdots \leq x(P\setminus I_{m-1})$.

Действительно, A) выполнено очевидным образом; Б$'$) выполнено, так как $\Aff F_{L'_k}=\Aff F_{L_k}$; В$''$) выполнено, так как $F_{L'_k}\subset\Aff F_{L_k}\cap H_{pq}^{\geq}$ и $F_{L_k}\cap F_{L'_k}=F_L$ ($p\in J_k\setminus I_{k-1}$, $q\in I_k\setminus J_k$ и $H_{pq}^{\geq}$ --- полупространство, заданное неравенством $x_p\geq x_q$).

Заметим, что согласно лемме \ref{LmapstoFL} и определению \ref{conjugatechain}, $L'_k$ и $L_k$ являются $I_k$-сопряженными и для грани $F_{L'_k}$ по построению выполнены условия (\ref{dFL}). Таким образом, мы доказали, что если многогранник $Q\in \K_{\P,\lambda}$ входит в $d(F_L)$ с ненулевым коэффициентом, то $Q=F_{L_k}$, где $L_k$ --- $\lambda$-допустимая цепь, полученная из $L$ некоторым $I_k$-уплотнением и имеющая $I_k$-сопряжённое, а именно $L_k'$. Причём $F_{L_k'}$ также входит в $d(F_L)$ с ненулевым коэффициентом.   

Таким образом, доказано, что все многогранники, входящие в $d(F_L)$ с ненулевым коэффициентом, входят и в правую часть формулы (\ref{phidelta}).
Обратно, пусть $L_k$ --- $\lambda$-допустимая цепь, полученная из $L$ некоторым $I_k$-уплотнением и имеющая $I_k$-сопряжённое (обозначим его через $L_k'$).
Благодаря проделанной работе легко видеть, что $F_{L_k}$ и $F_{L_k'}$ удовлетворяют условиям (\ref{dFL}).
Следовательно $F_{L_k}$ входит в $d(F_L)$ с ненулевым коэффициентом.
\end{proof}
%\section{Примеры}
%$\Bbb R^P$

\begin{example}\label{oldnewex}
Изменим немного чум из примера \ref{oplambda}.
Пусть $P=\{r,p,q,s,t\}$. Отношение частичного порядка $\preceq$ зададим накрывающими соотношениями: $r \precdot p \precdot q \precdot s$ (рис.\ \ref{mopmain}а -- крайний левый чум). В качестве $P^*$ выберем трёхэлементное подмножество $\{r, s, t\}$. Сохраняющую порядок функцию $\lambda: \P^* \rightarrow \R$ определим равенствами: $\lambda (r)=0$, $\lambda(s)=2$ и $\lambda(t)=1$ (рис.\ \ref{mopmain}а -- второй слева чум). Согласно определению (\ref{MOP}), 
\[
O(\P,\lambda)=
\begin{cases}
0=x_r \le x_p\le x_q  \le x_s=2,\\
x_t = 1.\\
\end{cases}
\]

Построим проекцию $O(\P,\lambda)$ на плоскость $Ox_px_q$ и проведём прямые $x_p=1$ и $x_q=1$, чтобы получить проекцию подразбиения $\K_{\P,\lambda}$ на ту же плоскость (см.\ замечание \ref{geometricKPlambda}, рис.\ \ref{mopmain}б).

\begin{figure}[h]
\begin{minipage}{0.7\linewidth}
\center{\includegraphics[width=0.8\linewidth]{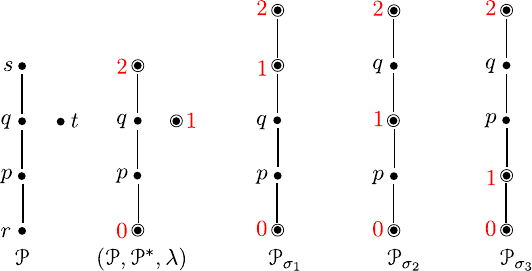} \\ а)}
\end{minipage}
\hfill
\begin{minipage}{0.29\linewidth}
\center{\includegraphics[width=\linewidth]{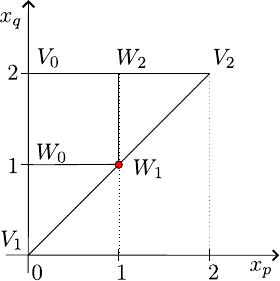} \\ б)}
\end{minipage}
\caption{а) Слева направо: диаграмма Хассе чума $\P$ из примера \ref{oldnewex}; диаграмма Хассе чума $\P$ из примера \ref{oldnewex}, в которой выделены элементы $P^*$ и указаны значения функции $\lambda$ на них; диаграммы Хассе линейных расширений $\P_{\sigma_1}$, $\P_{\sigma_2}$ и $\P_{\sigma_3}$ чума $\P$ из примера \ref{oldnewex}; б) проекция $\K_{\P,\lambda}$ на координатную плоскость $Ox_px_q$ для чума $\P$ из примера \ref{oldnewex}.}
\label{mopmain}
\end{figure} 

Вычислим $f$-вектор $O(\P,\lambda)$ с помощью теоремы \ref{main} и леммы \ref{lem1}. Для этого выпишем коцепной комплекс $C^*(\P,\lambda)$ и вычислим его когомологии. 

Мы знаем, что $C^i(\P,\lambda)$ --- векторное пространство над $\Z_2$ с базисом, состоящим из всех $\lambda$-допустимых цепей размерности $i$ 
в чуме порядковых идеалов чума $\P$. Выпишем базисные $\lambda$-допустимые цепи для каждого $i$.

Начнём с $\lambda$-допустимых цепей старшей размерности. Из доказательства предложения \ref{KP} мы знаем, что такие $\lambda$-допустимые цепи приходят из линейных расширений чума $\P$. В нашем примере $\lambda$-допустимые цепи максимальной размерности  соответствуют линейным расширениям $\P_{\sigma_1}$, $\P_{\sigma_2}$ и $\P_{\sigma_3}$, представленным на рис.\ \ref{mopmain}. Выпишем эти цепи и обозначим соответственно $L_{\sigma_1}$, $L_{\sigma_2}$ и $L_{\sigma_3}$, см.\ таблицу ниже. (Всюду далее при написании $\lambda$-допустимых цепей рассматриваемого чума $\P$ мы на месте отмеченных элементов будем писать значение функции $\lambda$ на них.)

\[\begin{array}{|c|c|c|}
\hline
L & \dim L& F_{L}\\
\hline
%L_{\sigma_1}=\Big(\emptyset \neq \{r\} \subsetneq \{r,p\} \subsetneq \{r,p,q\} \subsetneq \{r,p,q,t\} \subsetneq P\Big)& 2 & V_1W_1W_0\\

L_{\sigma_1}=\Big(\emptyset \neq \{0\} \subsetneq \{0,p\} \subsetneq \{0,p,q\} \subsetneq \{0,p,q,1\} \subsetneq P\Big)& 2 &  V_1W_1W_0: 0\leq x_p \leq x_q \leq 1 \\

\hline

%L_{\sigma_2}=\Big(\emptyset \neq \{r\} \subsetneq \{r,p\} \subsetneq \{r,p,t\} \subsetneq \{r,p,t,q\} \subsetneq P\Big)& 2 & W_0W_1W_2V_0\\

L_{\sigma_2}=\Big(\emptyset \neq \{0\} \subsetneq \{0,p\} \subsetneq \{0,p,1\} \subsetneq \{0,p,1,q\} \subsetneq P\Big)& 2 &W_0W_1W_2V_0: 0 \leq x_p \leq 1 \leq x_q \leq 2\\

\hline

%L_{\sigma_3}=\Big(\emptyset \neq \{r\} \subsetneq \{r,t\} \subsetneq \{r,t,p\} \subsetneq \{r,t,p,q\} \subsetneq P\Big)& 2 & W_1V_2W_2\\

L_{\sigma_3}=\Big(\emptyset \neq \{0\} \subsetneq \{0,1\} \subsetneq \{0,1,p\} \subsetneq \{0,1,p,q\} \subsetneq P\Big)& 2 &W_1V_2W_2: 1 \leq x_p \leq x_q \leq 2\\
\hline
\end{array}\]

Теперь выпишем $\lambda$-допустимые цепи на единицу меньшей размерности, см.\ таблицу ниже.
\[\begin{array}{|c|c|c|}
\hline
L & \dim L& F_{L}\\
\hline
%L_1=\Big(\emptyset \neq  \{r,p\} \subsetneq \{r,p,t\} \subsetneq \{r,p,t,q\} \subsetneq P\Big)& 1 & W_0V_0\\

L_1=\Big(\emptyset \neq  \{0,p\} \subsetneq \{0,p,1\} \subsetneq \{0,p,1,q\} \subsetneq P\Big)&1 & W_0V_0: 0 = x_p \leq 1 \leq x_q \leq 2\\

\hline
%L_2=\Big(\emptyset \neq \{r,p\} \subsetneq \{r,p,q\} \subsetneq \{r,p,q,t\} \subsetneq P\Big)& 1 & V_1W_0\\

L_2=\Big(\emptyset \neq   \{0,p\} \subsetneq \{0,p,q\} \subsetneq \{0,p,q,1\} \subsetneq P\Big)& 1 & V_1W_0: 0 = x_p \leq x_q \leq 1 \\

\hline

%L_3=\Big(\emptyset \neq \{r\} \subsetneq  \{r,p,q\} \subsetneq \{r,p,q,t\} \subsetneq P\Big)& 1 & V_1W_1\\

L_3=\Big(\emptyset \neq \{0\}   \subsetneq \{0,p,q\} \subsetneq \{0,p,q,1\} \subsetneq P\Big)&1 & V_1W_1: 0 \leq x_p = x_q \leq 1 \\

\hline

%L_4=\Big(\emptyset \neq \{r\} \subsetneq \{r,t\} \subsetneq  \{r,t,p,q\} \subsetneq P\Big)& 1 & W_1V_2\\

L_4=\Big(\emptyset \neq \{0\} \subsetneq \{0,1\} \subsetneq  \{0,1,p,q\} \subsetneq P\Big)& 1 & W_1V_2: 1 \leq x_p = x_q \leq 2\\

\hline

%L_5=\Big(\emptyset \neq \{r\} \subsetneq \{r,t\} \subsetneq \{r,t,p\} \subsetneq  P\Big)& 1 & V_2W_2\\

L_5=\Big(\emptyset \neq \{0\} \subsetneq \{0,1\} \subsetneq \{0,1,p\}  \subsetneq P\Big)&1 & V_2W_2: 1 \leq x_p \leq x_q = 2\\

\hline

%L_6=\Big(\emptyset \neq \{r\} \subsetneq \{r,p\} \subsetneq \{r,p,t\}  \subsetneq P\Big)& 1 & W_2V_0\\

L_6=\Big(\emptyset \neq \{0\} \subsetneq \{0,p\} \subsetneq \{0,p,1\}  \subsetneq P\Big)&1 & W_2V_0: 0 \leq x_p \leq 1 \leq x_q = 2\\

\hline

%L_7=\Big(\emptyset \neq \{r\}  \subsetneq \{r,p,t\} \subsetneq \{r,p,t,q\} \subsetneq P\Big)& 1 & W_1W_2\\

L_7=\Big(\emptyset \neq \{0\}  \subsetneq \{0,p,1\} \subsetneq \{0,p,1,q\} \subsetneq P\Big)&1 & W_1W_2: 0 \leq x_p = 1 \leq x_q \leq 2\\

\hline

%L_8=\Big(\emptyset \neq \{r\} \subsetneq \{r,p\} \subsetneq \{r,p,q,t\} \subsetneq P\Big)& 1 & W_1W_0\\

L_8=\Big(\emptyset \neq \{0\} \subsetneq \{0,p\} \subsetneq  \{0,p,q,1\} \subsetneq P\Big)&1 & W_1W_0: 0 \leq x_p \leq x_q = 1 \\

\hline
\end{array}\]

Осталось выписать нульмерные $\lambda$-допустимые цепи. Сделаем это.

\[\begin{array}{|c|c|c|}
\hline
L & \dim L& F_{L}\\

\hline

%L_{v_0}=\Big(\emptyset \neq \{r,p\} \subsetneq \{r,p,t\} \subsetneq  P\Big)& 0 & V_0\\

L_{v_0}=\Big(\emptyset \neq  \{0,p\} \subsetneq \{0,p,1\} \subsetneq  P\Big)&0 & V_0: 0 = x_p \leq 1 \leq x_q = 2\\

\hline
%L_{v_1}=\Big(\emptyset \neq  \{r,p,q\} \subsetneq \{r,p,q,t\} \subsetneq P\Big)& 0 & V_1\\

L_{v_1}=\Big(\emptyset \neq  \{0,p,q\} \subsetneq \{0,p,q,1\} \subsetneq P\Big)& 0 & V_1: 0 =x_p = x_q \leq 1 \\

\hline

%L_{v_2}=\Big(\emptyset \neq \{r\} \subsetneq \{r,t\} \subsetneq P\Big)& 0 & V_2\\

L_{v_2}=\Big(\emptyset \neq \{0\} \subsetneq \{0,1\} \subsetneq  P\Big)& 0 & V_2: 1 \leq x_p = x_q = 2\\

\hline
%L_{w_0}=\Big(\emptyset \neq  \{r,p\}  \subsetneq \{r,p,q,t\} \subsetneq P\Big)& 0 & W_0\\

L_{w_0}=\Big(\emptyset \neq \{0,p\} \subsetneq \{0,p,q,1\} \subsetneq P\Big)& 0 & W_0: 0 = x_p \leq x_q = 1 \\

\hline
%L_{w_1}=\Big(\emptyset \neq \{r\} \subsetneq \{r,p,q,t\} \subsetneq P\Big)& 0 & W_1\\

L_{w_1}=\Big(\emptyset \neq \{0\} \subsetneq \{0,p,q,1\} \subsetneq P\Big)& 0 & W_1: 0 \leq x_p = x_q = 1 \\

\hline

%L_{w_2}=\Big(\emptyset \neq \{r\}  \subsetneq \{r,p,t\} \subsetneq  P\Big)& 0 & W_2\\

L_{w_2}=\Big(\emptyset \neq \{0\}  \subsetneq \{0,p,1\} \subsetneq  P\Big)& 0 & W_2: 0 \leq x_p = 1 \leq x_q = 2\\

\hline
\end{array}\]

Таким образом, 
$$C^2(\P,\lambda)=\langle L_{\sigma_1}, L_{\sigma_2}, L_{\sigma_3}\rangle;$$
$$C^1(\P,\lambda)=\langle L_1, L_2, L_3, L_4, L_5, L_6, L_7, L_8\rangle;$$
$$C^0(\P,\lambda)=\langle L_{v_0}, L_{v_1},L_{v_2}, L_{w_0},
L_{w_1}, L_{w_2} \rangle.$$

Теперь вычислим коцепной гомоморфизм на базисных элементах групп $C^i(\P,\lambda)$. Будем действовать согласно правилу (\ref{delta}).

Рассмотрим $\lambda$-допустимую цепь $L_{v_0}=\Big(\emptyset \neq  \{0,p\} \subsetneq \{0,p,1\} \subsetneq  P\Big)$. Мы можем $I_1$-уплотнить $L_{v_0}$ единственным способом: получим $L_6=\Big(\emptyset \neq \{0\} \subsetneq \{0,p\} \subsetneq \{0,p,1\}  \subsetneq P\Big)$, а также $I_3$-уплотнить $L_{v_0}$ единственным способом: получим $L_1=\Big(\emptyset \neq  \{0,p\} \subsetneq \{0,p,1\} \subsetneq \{0,p,1,q\} \subsetneq P\Big)$. Таким образом, ни одно из $I_k$-уплотнений $L_{v_0}$ не имеет $I_k$-сопряжённое, а значит $\delta^0(L_{v_0})=0$.

Рассмотрим $\lambda$-допустимую цепь $L_{w_0}=\Big(\emptyset \neq \{0,p\} \subsetneq \{0,p,q,1\} \subsetneq P\Big)$. Вновь мы можем $I_1$-уплотнить рассматриваемую цепь только единственным образом: получим $L_8=\Big(\emptyset \neq \{0\} \subsetneq \{0,p\} \subsetneq  \{0,p,q,1\} \subsetneq P\Big)$, однако возможны два различных $I_2$-уплотнения. Сведения об $I_2$-уплотнениях $L_{w_0}$ занесём в небольшую таблицу:

\[\begin{array}{|c|c|c|}
\hline
\text{$I_2$-уплотнения $L_{w_0}=\Big(\emptyset \neq \{0,p\} \subsetneq \{0,p,q,1\} \subsetneq P\Big)$}
& J_2\setminus J_1& J_3\setminus J_2\\
\hline
L_1=\Big(\emptyset \neq  \{r,p\} \subsetneq \{r,p,1\} \subsetneq \{r,p,t,q\} \subsetneq P\Big)
&\{1\}&\{q\}\\
\hline
L_2=\Big(\emptyset \neq   \{0,p\} \subsetneq \{0,p,q\} \subsetneq \{0,p,q,1\} \subsetneq P\Big)
&\{q\}&\{1\}\\
\hline
\end{array}\]

Таким образом, $L_1$ и $L_2$, согласно определению \ref{conjugatechain}, являются $I_2$-сопряжёнными, поэтому $\delta^0(L_{w_0})=L_1+L_2$.

Рассмотрим $\lambda$-допустимую цепь $L_1=\Big(\emptyset \neq  \{0,p\} \subsetneq \{0,p,1\} \subsetneq \{0,p,1,q\} \subsetneq P\Big)$. Мы можем $I_1$-уплотнить $L_1$ единственным способом: получим $L_{\sigma_2}=\Big(\emptyset \neq \{0\} \subsetneq \{0,p\} \subsetneq \{0,p,1\} \subsetneq \{0,p,1,q\} \subsetneq P\Big)$. Таким образом, $L_1$ имеет единственное $I_k$ уплотнение и $\delta^1(L_1)=0$.

Рассмотрим $\lambda$-допустимую цепь $L_7=\Big(\emptyset \neq \{0\}  \subsetneq \{0,p,1\} \subsetneq \{0,p,1,q\} \subsetneq P\Big)$. Мы можем $I_2$-уплотнить её двумя различными способами. Сведения об $I_2$-уплотнениях $L_7$ занесём в небольшую таблицу:

\[\begin{array}{|c|c|c|}
\hline
\text{$I_2$-уплотнения $L_7=\Big(\emptyset \neq \{0\}  \subsetneq \{0,p,1\} \subsetneq \{0,p,1,q\} \subsetneq P\Big)$}
& J_2\setminus J_1& J_3\setminus J_2\\
\hline
L_{\sigma_2}=\Big(\emptyset \neq \{0\} \subsetneq \{0,p\} \subsetneq \{0,p,1\} \subsetneq \{0,p,1,q\} \subsetneq P\Big)
&\{p\}&\{1\}\\
\hline
L_{\sigma_3}=\Big(\emptyset \neq \{0\} \subsetneq \{0,1\} \subsetneq \{0,1,p\} \subsetneq \{0,1,p,q\} \subsetneq P\Big)
&\{1\}&\{p\}\\
\hline
\end{array}\]

Таким образом, $L_{\sigma_2}$ и $L_{\sigma_3}$, согласно определению \ref{conjugatechain}, являются $I_2$-сопряжёнными, поэтому $\delta^1(L_7)=L_{\sigma_2}+L_{\sigma_3}$.

Для остальных базисных элементов групп $C^i(\P,\lambda)$ вычисления можно выполнить аналогично. Читателю предоставляется возможность провести их самостоятельно  и сравнить полученные результаты с теми, которые выписаны в таблице: 

\[\begin{array}{|c|c|c|c|c|c|}
\hline
C^0(\P,\lambda)&C^1(\P,\lambda)&C^2(\P, \lambda)&\delta^0&\delta^1\\
\hline
L_{v_0} &L_1&L_{\sigma_1}&\delta^0(L_{v_0})=0&\delta^1(L_1)=0\\
\hline
 L_{v_1} &L_2&L_{\sigma_2}&\delta^0(L_{v_1})=0&\delta^1(L_2)=0\\
\hline
 L_{v_2} &L_3&L_{\sigma_3}&\delta^0(L_{v_2})=0& \delta^1(L_3)=0\\
\hline
L_{w_0} &L_4&-&\delta^0(L_{w_0})=L_1+L_2&\delta^1(L_4)=0\\
\hline
L_{w_1} &L_5&-&\delta^0(L_{w_1})=L_3+L_4&\delta^1(L_5)=0\\
\hline
 L_{w_2} &L_6&-&\delta^0(L_{w_2})=L_5+L_6&\delta^1(L_6)=0\\
\hline
-&L_7& -&-&\delta^1(L_7)=L_{\sigma_2}+L_{\sigma_3}\\
\hline
-&L_8&-&-&\delta^1(L_8)=L_{\sigma_1}+L_{\sigma_2}\\
\hline
\end{array}\]

\[
H^0(C^*(\P,\lambda))\simeq \ker \delta^0\simeq \langle L_{v_0}, L_{v_1}, L_{v_2}\rangle \simeq \Z_2^3,
\]
\[
H^1(C^*(\P,\lambda)) \simeq\langle  L_1,\dots, L_6\rangle/\langle L_1+L_2,\ L_3+L_4,\ L_5+L_6\rangle \simeq \Z_2^3,
\]
\[
H^2(C^*(\P,\lambda))\simeq\langle L_{\sigma_1},L_{\sigma_2}, L_{\sigma_3} \rangle/\langle L_{\sigma_1}+L_{\sigma_2},\ L_{\sigma_2}+L_{\sigma_3} \rangle \simeq \Z_2.
\]
Согласно лемме \ref{lem1} и теореме \ref{main}  
 \[f_0(O(\P,\lambda))=\dim H^0(C^*(\P,\lambda))=3,\]
 \[f_1(O(\P,\lambda))=\dim H^1(C^*(\P,\lambda))=3,\]
 \[f_2(O(\P,\lambda))=\dim H^2(C^*(\P,\lambda))=1.\]

В завершении данного примера заметим, что коцепной комплекс $C^*_{\K_{\P,\lambda}}$ для рассматриваемого многогранника $O(\P,\lambda)$ мы выписали в примере \ref{firstexample}. Несложно видеть, что коцепные комплексы примеров \ref{firstexample} и \ref{oldnewex} связывает изоморфизм, построенный в доказательстве теоремы \ref{main}.
\end{example}

\section*{Благодарности}

Я признательна моему научному руководителю В.~А.~Тиморину за общую постановку задачи, постоянное внимание и поддержку, М.~Э.~Казаряну, В.~А.~Кириченко, С.~К.~Ландо и А.~И.~Эстерову за интерес к работе. Я очень благодарна моему мужу Серёже Мелихову за поддержку и веру в меня. Также хотелось бы поблагодарить рецензента за полезные замечания.

%Е. Ю. Смирнов

\end{document}